\documentclass[12pt] {article}
\usepackage{amsmath,amsthm}
\usepackage{amssymb,latexsym}

\usepackage{comment}
\usepackage{empheq}
\usepackage{amscd}
\usepackage{chngcntr}
\usepackage{enumitem}
\usepackage{tikz-cd}
\usepackage{bm}

\newlist{paragraphlist}{enumerate}{1}

\setlist[paragraphlist,1]{leftmargin=*,label={\bfseries \arabic*}}

\counterwithin{paragraphlisti}{subsubsection}

\title{Towards Hodge-Riemann relations for non-Archimedean analogs of valuations on convex sets.}
\date{}
\author{Semyon Alesker
\\  { \normalsize Department of Mathematics, Tel Aviv University, Ramat Aviv}
\\  { \normalsize 69978 Tel Aviv, Israel }
\\ {\normalsize e-mail: semyon@tauex.tau.ac.il }}

\def\RR{\mathbb{R}}
\def\CC{\mathbb{C}}
\def\QQ{\mathbb{Q}}

\def\ZZ{\mathbb{Z}}

\def\PP{\mathbb{P}}

\def\FF{\mathbb{F}}
\def\SS{\mathbb{S}}

\def\alp{\alpha}
\def\ome{\omega}

\def\lam{\lambda}
\def\Lam{\Lambda}

\def\to{\longrightarrow}

\def\qed { Q.E.D. }

\def\inj{\hookrightarrow}

\swapnumbers
\newtheorem{theorem}{Theorem}[section]

\newtheorem{lemma}[theorem]{Lemma}
\newtheorem{proposition}[theorem]{Proposition}
\newtheorem{claim}[theorem]{Claim}
\theoremstyle{definition}

\newtheorem{definition}[theorem]{Definition}
\newtheorem{remark}[theorem]{Remark}
\theoremstyle{conjecture}
\newtheorem{conjecture}[theorem]{Conjecture}
\theoremstyle{principle}

\def\ca{{\cal A}}  
\def\cd{{\cal D}}  
  
 \def\ck{{\cal K}} \def\cl{{\cal L}}
  \def\co{{\cal O}}
\def\cp{{\cal P}} \def\cq{{\cal Q}}

\def\bbx{\textbf{x}}
\def\bby{\textbf{y}}
\def\bbz{\textbf{z}}
\def\bbe{\textbf{e}}
\def\bbf{\textbf{f}}
\def\bbv{\textbf{v}}
\def\bbw{\textbf{w}}

\numberwithin{equation}{section}
\begin{document}
\maketitle

\begin{abstract}
In \cite{alesker-bernstein}, a non-Archimedean analogue of the space of translation-invariant even valuations on convex sets was introduced. In \cite{alesker-non-arch}, motivated by a further analogy with the classical theory, this space was equipped with two multiplicative structures, the product and the convolution. Both structures satisfy Poincar\'e duality and the (non-mixed) hard Lefschetz theorem.

In this paper, we formulate a conjecture concerning more general mixed versions of the hard Lefschetz theorem and the Hodge--Riemann relations. We prove the non-mixed Hodge--Riemann relations in degree $1$ for the product and, equivalently, in codegree $1$ for the convolution.
\end{abstract}

\tableofcontents

\section{Introduction}\label{S:intro}
\begin{paragraphlist}

\item A K\"ahler package is an algebra over the real or complex numbers with a finite (bi)grading that satisfies Poincar\'e duality, the hard Lefschetz theorem, and the Hodge--Riemann relations.

We state below a general definition of a K\"ahler package in the commutative graded setting. Although this definition does not encompass many examples arising in algebraic geometry, it is sufficient for our purposes.

K\"ahler packages are typically rare, and proving their existence is often highly non-trivial.

Originally, examples of K\"ahler packages were provided by the cohomology algebras of compact K\"ahler (for example, complex projective) manifolds; see, for example, \cite{griffiths-harris}.

Later, examples of a different nature were discovered in combinatorics and convex geometry, in particular in the theory of valuations on convex sets.

 Relevant work in combinatorics includes \cite{stanley, mcmullen-89, mcmullen-93, karu-2004, adiprasito-huh-katz, huh-icm}.
 In some of these works, long-standing conjectures in combinatorics were solved by constructing a K\"ahler package or some of its constituent structures.

The K\"ahler packages constructed in the theory of valuations on convex sets strongly influenced the present work. Several constructions used in this paper have direct analogues in that theory.
However, from a technical point of view, the present paper is independent of that theory, and no prior knowledge of it is required. Relevant analogies from the theory of valuations on convex sets will be mentioned below as needed.

The new results of this work are Conjecture \ref{E:conject}, which formulates the existence of the K\"ahler package in a new setting, and Theorem \ref{T:thm}, which establishes the Hodge--Riemann relations in a special case of this conjecture.
Below, we review the relevant background.

\item Let $$A=A_0\oplus A_1\oplus \dots\oplus A_n$$ be a commutative graded algebra over $\RR$, i.e.,
\begin{eqnarray}\label{E:garded-alg}
A_i\cdot A_j\subset A_{i+j}.
\end{eqnarray}
Assume that $A_0=\RR$ and $\dim A_n=1$. We also assume that $A_n$ is oriented, so that the notion of a positive element of $A_n$ is well defined.

\item One says that $A$ satisfies Poincar\'e duality if the bilinear pairing
$$A_i\times A_{n-i}\to A_n$$
given by multiplication is non-degenerate. That is, for every nonzero element $x\in A_i$, there exists $y\in A_{n-i}$ such that $x\cdot y\ne 0$.
\begin{remark}
If $A$ is finite dimensional, Poincar\'e duality implies that
$$\dim A_i=\dim A_{n-i}.$$
\end{remark}

\item One says that $A$ satisfies the mixed hard Lefschetz property if there exists a non-empty subset $\ck\subset A_1$ such that, for every $i<n/2$ and every choice of
$\ome_1,\dots,\ome_{n-2i}\in\ck$, the map $A_i\to A_{n-i}$ given by
$$x\mapsto x\cdot \ome_1\cdot\dots\cdot \ome_{n-2i}$$
is an isomorphism.

If this property holds only in the special case $\ome_1=\dots=\ome_{n-2i}$, then one says that $A$ satisfies the non-mixed hard Lefschetz property.

\begin{remark}
If $A$ is finite-dimensional, then the hard Lefschetz property, whether mixed or non-mixed, implies that
$$1=\dim A_0\leq \dim A_1\leq \dots \leq \dim A_{[n/2]}.$$
Applications of the hard Lefschetz property to combinatorics often proceed along these lines; see, for example, \cite{stanley}.
\end{remark}

\item One says that $A$ satisfies the mixed Hodge--Riemann relations if the following condition holds. Let $i\leq n/2$, and let $\ome_1,\dots,\ome_{n-2i+1}\in \ck$. Define the subspace $P_i\subset A_i$ of primitive elements by
\begin{eqnarray}\label{E:primitive}
P_i=\{x\in A_i|\,  x\cdot \ome_1\cdot\dots\cdot \ome_{n-2i}\ome_{n-2i+1}=0\}.
\end{eqnarray}
Consider the quadratic form on $A_i$ with values in $A_n$:
\begin{eqnarray}\label{E:form-Q}
Q(x)=(-1)^ix^2\cdot \ome_1\cdot\dots\cdot \ome_{n-2i}.
\end{eqnarray}
One says that $A$ satisfies the mixed Hodge--Riemann relations if the restriction of $Q$ to $P_i$ is positive definite; that is,
\begin{eqnarray}\label{E:HR-gen}
Q(x)>0 \mbox{ for any } x\in P_i\backslash\{0\}.
\end{eqnarray}

One says that $A$ satisfies the non-mixed Hodge--Riemann relations if the above property holds in the special case
$$\ome_1=\dots=\ome_{n-2i}=\ome_{n-2i+1}.$$

 In this paper, we will work with algebras over $\CC$ obtained by complexifying a real graded algebra. In this setting, the definition (\ref{E:form-Q}) of the quadratic form $Q$ should be replaced by
$$Q(x)=(-1)^ix\cdot \bar x \cdot \ome_1\cdot\dots\cdot \ome_{n-2i},$$
where $\bar x$ denotes the complex conjugate of $x$. This defines a Hermitian form.

\begin{remark}
\begin{enumerate}
\item The well-known Khovanskii--Teissier argument derives the Alexandrov--Fenchel inequality for mixed volumes of convex bodies from a special case of the non-mixed Hodge--Riemann relations for toric varieties; see, for example, \cite{burago-zalgaller}, \S 27.

\item The mixed Hodge--Riemann relations for valuations on convex sets also imply the Alexandrov--Fenchel inequality; see \cite{kotrbaty-conj}, Corollary 8.1. They further imply new inequalities for mixed volumes of convex bodies; see \cite{alesker-isr-ineq, kotrbaty-wannerer-AF}.

\item For applications of the non-mixed Hodge--Riemann relations to the combinatorics of matroids, see \cite{adiprasito-huh-katz}.
\end{enumerate}
\end{remark}

\item Let us provide some background specific to the present paper.
Let $\FF$ be a non-Archimedean local field (see Section \ref{S:local-fields}), and let $V$ be an $n$-dimensional vector space over $\FF$. In \cite{alesker-bernstein}, J. Bernstein and the author introduced a graded vector space
$$Val^\infty(V)=\oplus_{i=0}^n Val^\infty_i(V),$$
where $Val^\infty_0(V)$ and $Val^\infty_n(V)$ are one-dimensional, while all other summands are infinite-dimensional. In fact, $Val^\infty_0(V)=\CC$, and $Val^\infty_n(V)$ is the space of $\CC$-valued Haar measures on $V$, also denoted by $D(V)$.

The space $Val^\infty_i(V)$ is an irreducible representation of the group $GL(V)$ of invertible linear transformations of $V$.
It was introduced in \cite{alesker-bernstein} as the image of certain integral operators between spaces of sections of line bundles over appropriate Grassmannians of linear subspaces of $V$; see Section \ref{S:valuations}, paragraph \ref{i:def-valuations}. The space $Val^\infty(V)$ is intended to be an analogue of the space of translation-invariant even smooth valuations on convex sets.

In \cite{alesker-non-arch}, the author introduced two multiplicative structures: the product $\cdot$ on $Val^\infty(V)$ and the convolution $\ast$ on $Val^\infty(V)\otimes D(V)^*$. The construction of the product was motivated by the analogous product on translation-invariant smooth valuations on convex sets introduced by the author in \cite{alesker-04}, while the construction of the convolution was motivated by the analogous convolution on the corresponding space of valuations on convex sets introduced by Bernig and Fu in \cite{bernig-fu-convol}.

$(Val^\infty(V),\cdot)$ is a graded algebra in the sense of (\ref{E:garded-alg}).
The convolution satisfies an analogous property, with degree replaced by codegree.
Both the product and the convolution satisfy Poincar\'e duality \cite{alesker-non-arch}, just as in the case of valuations on convex sets \cite{alesker-04, bernig-fu-convol}.

The product and the convolution are related by a Fourier-type transform. More precisely, the author constructed in \cite{alesker-non-arch} an isomorphism of algebras
\begin{eqnarray}\label{E:fourier}
\FF\colon (Val^\infty(V),\cdot)\tilde\to (Val^\infty(V^\vee)\otimes D(V^\vee)^*,\ast),
\end{eqnarray}
where $V^\vee$ denotes the dual space of $V$. The isomorphism $\FF$ commutes with the action of $GL(V)$ on both spaces.
Consequently, any property of the K\"ahler package established for either the product or the convolution immediately implies the corresponding property for the other operation. For this reason, we will continue the detailed discussion only for the product.

For valuations on convex sets, the Fourier transform was constructed by the author in \cite{alesker-jdg-03, alesker-fourier}. A different and simpler construction was later given by Faifman and Wannerer \cite{faifman-wannerer}.

For any lattice $\Lam\subset V$ (Definition \ref{D:lattice}), the subspace of elements of $Val^\infty_i(V)$ fixed by  the subgroup of transformations preserving $\Lam$,
$$GL(\Lam)=\{g\in GL(V)\mid g(\Lam)=\Lam\}$$
is one dimensional. We denote by $V_{\Lam,i}$ the unique generator of this subspace satisfying condition (\ref{E:V-Lam-i}) below.

 Let $\ck\subset Val_1^\infty(V)$ denote the subset consisting of all elements of the form $V_{\Lam,1}$, where $\Lam$ ranges over all lattices in $V$.

The author proved in \cite{alesker-non-arch} that $Val^\infty(V)$ satisfies the non-mixed hard Lefschetz theorem for the product. For the convolution, the analogous statement holds with $V_{\Lam,1}$ replaced by $V_{\Lam,n-1}$.

\item We now state the new results of this paper. We begin with a conjecture formulated for the product; the corresponding formulation for the convolution is immediate.
\begin{conjecture}\label{E:conject}
\begin{enumerate}
\item\label{L:conj1} $Val^\infty(V)$ satisfies the mixed hard Lefschetz theorem for the subset $\ck$ as above.

\item\label{L:conj2} $Val^\infty(V)$ satisfies the mixed Hodge-Riemann relations for the subset $\ck$ as above.
\end{enumerate}
\end{conjecture}
The non-mixed case of part \ref{L:conj1} of the conjecture was proved by the author in \cite{alesker-non-arch}.
The main new result of this paper is a proof of the second part of this conjecture in the special case of the non-mixed Hodge--Riemann relations for $i=1$.
\begin{theorem}\label{T:thm}
Let $\dim V=n\geq 2$.
For any lattice $\Lam\subset V$, the quadratic form
$$Q(\phi)=-\phi\cdot\bar\phi\cdot (V_{\Lam,1})^{n-2}$$
is positive definite on the subspace
$$P_1=\{\phi\in Val_1^\infty(V)\mid \phi\cdot (V_{\Lam,1})^{n-1}=0\}.$$
\end{theorem}

\item Conjecture \ref{E:conject} is motivated by the analogy with the theory of valuations on convex sets. Special cases of the hard Lefschetz theorem were proved by the author in \cite{alesker-jdg-03, alesker-hard-Lefschetz} and by Bernig and Br\"ocker \cite{bernig-broecker}.
More recently, Kotrbat\'y \cite{kotrbaty-conj} conjectured a much more general version of the mixed hard Lefschetz theorem. In the same paper, he also conjectured mixed Hodge--Riemann relations for valuations on convex sets and proved them in a special non-mixed case. Subsequently, Kotrbat\'y and Wannerer \cite{kotrbaty-wannerer} proved another non-mixed case.
Finally, Bernig, Kotrbat\'y, and Wannerer \cite{bernig-kotrbaty-wannerer} proved Kotrbat\'y's conjecture in full for valuations on convex sets.

\item The general strategy of the proof of Theorem \ref{T:thm} is motivated by Kotrbat\'y's proof \cite{kotrbaty-conj} of a special case of the non-mixed Hodge--Riemann relations for even valuations on convex sets.
The quadratic form $Q$ can be expressed in terms of a composition of certain integral transforms, called the Radon and sine transforms, acting on functions on Grassmannians over the field $\FF$; see Lemma \ref{L:form-necessary}.\footnote{Lemma \ref{L:form-necessary} is analogous to a formula from \cite{alesker-hard-Lefschetz}  for valuations on convex sets.} This reduces the proof of
Theorem \ref{T:thm} to harmonic analysis on Grassmannians. The relevant operators commute with the action of the maximal compact subgroup $GL(\Lam)$ of the full linear group $GL(V)$. The action of $GL(\Lam)$ on functions on the Grassmannians is multiplicity-free. Therefore, by Schur's lemma, the composition of the Radon and sine transforms times $(-1)^i$ acts by a scalar on every irreducible subrepresentation. To prove the main result, one has to verify that the corresponding eigenvalues are positive on the irreducible subspaces corresponding to primitive valuations. In this paper, we carry out this verification in the special case of the Grassmannian of lines, that is, the projective space.
Our main tool for this step is the theory of spherical harmonics over non-Archimedean local fields, developed recently by P. Humphries \cite{humphries}.

\item {\bf Acknowledgement.} I thank P. Humphries for useful correspondence on MathOverflow, in particular for bringing \cite{humphries} to my attention.

\item {\bf AI disclosure.} The author used ChatGPT for proofreading. The mathematical content is entirely the author’s own, and the author takes full responsibility for the paper.

\end{paragraphlist}

\section{Reminder on local fields.}\label{S:local-fields}
\begin{paragraphlist}
\item In this section we collect a few basic well known facts on local fields sufficient for this paper. We refer to \cite{weil}, Ch. 1, for details.

By definition, a local field is a topological locally compact non-discrete field such that the operations of addition, multiplication, and division are continuous in this topology.
There is a classification of such fields: they are precisely $\RR,\CC, \FF_q((t))$, and finite extensions of the
fields of $p$-adic numbers $\QQ_p$. Here $\FF_q$ denotes the finite field with $q$ elements, and $\FF_q((t))$ denotes the field of formal Laurent power series.
The first two examples, namely $\RR,\CC$, are called Archimedean, while all others are called non-Archimedean local fields.

\item Let $\FF$ be a non-Archimedean local field. It has a unique maximal compact subring $\co\subset \FF$. For example if $\FF=\FF_q((t))$ then $\co=\FF_q[[t]]$
is the ring of all formal power series.
If $\FF=\QQ_p$ then $\co=\ZZ_p$ is the ring of $p$-adic integers.

The field of fractions of $\co$ equals $\FF$.

\item $\co$ has a unique maximal ideal $\frak{m}\subset \co$. It is generated by a single element which will be fixed throughout the paper and denoted by $\pi$.
For example for $\FF=\FF_p((t))$ the ideal $\frak{m}$ is generated by $t$,
while for $\FF=\QQ_p$ the ideal $\frak{m}$ is generated by $p$.

The quotient $k:=\co/\frak{m}$ is necessarily a finite field; it is called the residue field of $\FF$. Denote by $q:=|k|$ its number of elements.

\item There exists a unique multiplicative norm
$$ |\cdot|\colon \FF\to \RR_{\geq 0}$$
such that
\begin{eqnarray*}
|x|=1 \,\,\, \forall x\in\co\backslash \frak{m},\\
|x|=\frac{1}{|k|^{i}} \,\,\, \forall x\in \frak{m}^i\backslash \frak{m}^{i+1}, \mbox{where } i\geq 1,
\end{eqnarray*}
where $|k|$ denotes the cardinality of the residue field.
Multiplicativity means that  $|x\cdot y| =|x|\cdot |y|$ for any $x,y\in \FF$.

This norm satisfies the strengthened triangle inequality
$$|x+y|\leq \max\{|x|,|y|\}.$$

\item The norm $|\cdot|$ has the following property. Let $\mu$ be a Lebesgue measure on $\FF$ ($\mu$ exists and is unique up to a proportionality). Let $x\in \FF$.
Then $$\mu(x\cdot A)=|x|\mu(A)$$
for any compact subset $A\subset \FF$.

\end{paragraphlist}

\section{Lattices over non-Archimedean local fields.}\label{S:lattice}
\begin{paragraphlist}
\item In this section $\FF$ denotes a non-Archimedean local field, and $\co\subset \FF$ its ring of integers. In this section we review, mostly following \cite{weil}, a few well known facts on finite dimensional $\FF$-vector spaces
and lattices in them.

A proof of the following result can be found in \cite{schaeffer}, Thm. 3.2, Ch. 1.
\begin{theorem}\label{T:isomorphism-finite-dim}
Let $V$ be an $n$-dimensional Hausdorff topological vector space over a local field $\FF$. Let $v_1,\dots,v_n$ be its basis. Then the map $\FF^n\to V$ given by
$$(x_1,\dots,x_n)\mapsto x_1v_1+\dots+x_nv_n$$
is an isomorphism of topological vector spaces when the source space is equipped with the product topology.
\end{theorem}

\item Let $V$ be an $n$-dimensional Hausdorff topological vector space over the local field $\FF$.
\begin{definition}\label{D:lattice}
A lattice $\Lam$ in $V$ is a compact open $\co$-submodule of $V$.
\end{definition}

 \begin{lemma}\label{L:sublattice-quotient}
Let $\Lam\subset V$ be a lattice. Let $E\subset V$ be a vector subspace. Then $E\cap \Lam$ is a lattice in $E$, and $\Lam/(E\cap \Lam)$ is a lattice in $V/E$.
\end{lemma}
{\bf Proof.} This immediately follows from Definition \ref{D:lattice}. \qed

\begin{theorem}[\cite{weil}, Ch. II, \S 2, Thm. 1]\label{T:free-lattice}
Let $\Lam\subset V$ be a lattice.
\newline
(i) Then $V$ has a basis $v_1,\dots, v_n$ such that $\Lam=\co v_1\oplus\dots \oplus \co v_n$. In particular, $\Lam$ is a free module of rank $n$.
\newline
(ii) Moreover, if $\{0\}=V_0\subset V_1\subset \dots\subset V_{n-1}\subset V_n=V$ is a sequence of linear subspaces such that $\dim V_i=i$, then the above vectors $ v_1,\dots, v_n$
can be chosen so that $v_{1},\dots,v_i$ is a basis of $V_i$ for any $i$.
\end{theorem}

\begin{remark}\label{Rem:basis-of-lattice}
Under the assumptions of part (ii) of the last theorem, one clearly has for each $i$
$$\Lam\cap V_i=\co v_1\oplus\dots\oplus \co v_i.$$
\end{remark}

\item Let $\Lam\subset V$ be a lattice. Denote by $GL(\Lam)$ the subgroup
$$GL(\Lam):=\{T\in GL(V)|\, T(\Lam)=\Lam\}.$$

Let $Gr_i^V$ denote the Grassmannian of $i$-dimensional subspaces of $V$.

\begin{proposition}\label{P:transitivity}
Let $\Lam\subset V$ be a lattice. The natural action of the group $GL(\Lam)\simeq GL_n(\co)$ on the Grassmannian $Gr_i^V$ is transitive.
\end{proposition}
For a proof see, e.g., \cite{alesker-non-arch}, Proposition 4.6.

\end{paragraphlist}

\section{Grassmannians.}\label{S:grassmannians.}
\begin{paragraphlist}
\item The material of this section is folklore. We describe the Grassmannian over a commutative local ring rather than a field.

\item Let $A$ be a commutative ring with a unit.  Recall $A$  is called local if it has a unique maximal ideal $\frak{p}$ different from $A$.
The quotient ring $A/\frak{p}$ is necessarily a field, called the residue field, and will be denoted by
$$\kappa=A/\frak{p}.$$



\item Let $(A,\frak{p})$ be a local ring. We denote by $Gr_{k,n}(A)$ the set of free $A$-submodules $\cp\subset A^n$ of rank $k$, i.e., $\cp\simeq A^k$,  and such that there exists
a free submodule $\cq\subset A^n$ such that $\cp\oplus \cq=A^n$.

Note that necessarily  $\cq$ is a free module of rank $n-k$.
\begin{remark}\label{Rem:complement-mod}
\begin{enumerate}
\item If $A$ is a field, then $Gr_{k,n}(A)$ is the usual set of linear $k$-dimensional subspaces.

\item It is well known that if $(A,\frak{p})$ is a local ring, then if $\cp,\cq\subset A^n$ are submodules such that
$$A^n=\cp\oplus \cq$$ then $\cp$ and $\cq$ are free, and the sum of their ranks equals $n$; see, e.g., Exercise 4.11(a) in \cite{eisenbud-book}.
The main examples used here are the rings of integers $\co$ in a local field $\FF$ and its quotients  by proper ideals; these are local rings.

\item The above definition of the Grassmannian coincides with the standard one in the case of local rings, which is the only case of interest in this paper.
\end{enumerate}
\end{remark}

\item If $B$ is a non-zero quotient ring of $A$, then we have the canonical map of sets
\begin{eqnarray}\label{E:grass-map}
\iota\colon Gr_{k,n}(A)\to Gr_{k,n}(B).
\end{eqnarray}
Indeed, a direct summand $\cp\subset A^n$  is mapped to the direct summand $\cp\otimes_AB\subset B^n$.

\item The following lemma is well known folklore. We present a proof due to lack of a precise reference.
\begin{lemma}\label{L:two-Grassmannians}
Let $\FF$ be a non-Archimedean local field. Let $\co$ be its ring of integers. Then the natural map $Gr_{k,n}(\co)\to Gr_{k,n}(\FF)$ given by
$$M\mapsto M\otimes_\co \FF$$
is a bijection. The inverse map is given by $L\mapsto L\cap \co^n$.
\end{lemma}
{\bf Proof.} It suffices to show that the two compositions of these maps are equal to the identity.

Let $M\subset \co^n$ be an $\co$-submodule which belongs to $Gr_{k,n}(\co)$.
Then $M\simeq \co^k$, and there exists an $\co$-submodule $M'\subset \co^n$ such that $$M\oplus M'=\co^n$$
and by Remark \ref{Rem:complement-mod}(2) it is isomorphic to $\co^{n-k}$.
Choosing bases of $M$ and $M'$ we may assume that
$$M=\co e_1\oplus \dots\oplus \co e_{k},\, M'=\co e_{k+1}\oplus\dots \oplus \co e_{n}.$$
Then $$M\otimes_\co \FF=\FF e_1\oplus\dots \oplus \FF e_{k},\, M'\otimes_\co \FF=\FF e_{k+1}\oplus\dots \oplus \FF e_{n}.$$
Now it is clear that $$(M\otimes_\co \FF)\cap \co^n=M.$$
That proves that the first composition of the relevant maps equals the identity.

Let us consider the opposite composition. Namely, let $L\subset\FF^n$ be a $k$-dimensional vector subspace. We have to show that
\begin{eqnarray}\label{L:inverse-comp}
(L\cap\co^n)\otimes_\co\FF=L.
\end{eqnarray}

By Theorem \ref{T:free-lattice}(ii) there exists a basis $v_1,\dots,v_n$ of $\FF^n$ such that $L=span_\FF\{v_1,\dots,v_k\}$ and
$$\co^n=\co v_1\oplus \dots \oplus \co v_n.$$
Then obviously $$L\cap\co^n=\co v_1\oplus \dots \oplus \co v_k.$$
This implies (\ref{L:inverse-comp}). \qed

\item \begin{lemma}\label{E:maps-grassm-local}
Let $(A,\frak{p})$ be a commutative local ring with a unit. Let $\frak{q}\subset \frak{p}$ be an ideal. Then  the map (\ref{E:grass-map})
$$\iota\colon Gr_{k,n}(A)\to Gr_{k,n}(A/\frak{q})$$
is onto.
\end{lemma}
{\bf Proof.} We have the obvious map
\begin{eqnarray}\label{E:map-modulo}
 A^n\to (A/\frak{q})^n.
\end{eqnarray}
Assume we have a direct sum decomposition
$$(A/\frak{q})^n=\bar\cp\oplus \bar\cq,$$
where $\bar\cp$ and $ \bar\cq$ are free of ranks $k$ and $n-k$ respectively.

Let us choose bases $\{\bar p_a\}_{a=1}^k$ of $\bar\cp$ and $\{\bar q_b\}_{b=k+1}^{n}$ of $\bar\cq$. Let $\{p_a\}$ and $\{q_b\}$ be their lifts to $A^n$.
Let $\cp$ be the $A$-submodule of $A^n$ generated by $\{p_a\}$, and $\cq$ be the submodule generated by $\{q_b\}$.

 We claim that
\begin{eqnarray}\label{E:p-plus-q}
\cp\oplus \cq =A^n,\, \cp \simeq A^k,
\end{eqnarray}
and the image of $\cp$ under the map (\ref{E:map-modulo}) is equal to $\bar \cp$.

The last claim is obvious.

The $n$ elements of $\kappa^n$ obtained by reduction of $\{p_a\}_{a=1}^k$ and $\{q_b\}_{b=k+1}^{n}$ modulo $\frak{p}$ form a basis of $\kappa^n$.

Let $T$ be an $n\times n$ matrix with coefficients in $A$ whose columns are  $\{p_a\}_{a=1}^k$ and $\{q_b\}_{b=k+1}^{n}$.
Its reduction modulo $\frak{p}$ has linearly independent columns over $\kappa$.
Hence $\det T$ is invertible in $A$. Hence the map $\tilde T\colon A^n\to A^n$ given by multiplication by $T$
$$x\mapsto Tx$$
is an isomorphism of $A$-modules. The inverse isomorphism satisfies
$$\tilde T^{-1}(p_a)=e_a,\, \tilde T^{-1}(q_b)=e_b,$$
where $e_1,\dots,e_n$ is the standard basis of $A^n$. This implies (\ref{E:p-plus-q}).  \qed

\item For a local ring $A$, the group $GL_n(A)$ acts naturally on $Gr_{k,n}(A)$:
$$(g,P)\mapsto g(P).$$
This action is obviously transitive.
In particular, $GL_n(\co/\frak{m}^l)$ acts transitively on $Gr_{k,n}(\co/\frak{m}^l)$. The ring homomorphism
$$\co\to \co/\frak{m}^l$$
induces the obvious group homomorphism
$$GL_n(\co)\to GL_n(\co/\frak{m}^l).$$
This homomorphism defines an action of $GL_n(\co)$ on all $Gr_{k,n}(\co/\frak{m}^l)$.

It is easy to see that under this action all the maps
$$\iota_l\colon Gr_{k,n}(\co/\frak{m}^{l+1})\to Gr_{k,n}(\co/\frak{m}^{l})$$
are $GL_n(\co)$-equivariant. Hence the projective limit $$ \lim_{\underset{l}{\leftarrow}}Gr_{k,n}(\co/\frak{m}^l)$$
of finite discrete spaces, equipped with the projective limit topology, carries a continuous action of the group $GL_n(\co)$.

\item Let us describe a stabilizer in $GL_n(\co)$ of a point in the Grassmannian $Gr_{k,n}(\co/\frak{m}^l)$. Consider the submodule of $(\co/\frak{m}^l)^n$
spanned by the first $k$ standard vectors. One can easily check the following lemma.
\begin{lemma}\label{L:stabilizer}
Let $l\geq 1$.
The stabilizer of the above submodule in $GL_n(\co)$ is equal to the subgroup
$$P_k(\frak{m}^l)= \left\{\left[ \begin{array}{c|c}
                   A&b\\
                   \hline
                   c&d
                   \end{array}\right]\in GL_n(\co)\Big|\, A\in GL_k(\co),\, c\in Mat_{n-k\times k}(\frak{m}^l),\, d\in GL_{n-k}(\co)\right\},$$

where $Mat_{n-k\times k}(\frak{m}^l)$ denotes the space of $(n-k)\times k$ matrices with entries in $\frak{m}^l$.
\end{lemma}

\item By Proposition \ref{P:transitivity}, the group $GL_n(\co)$ acts transitively on the Grassmannian $Gr_{k,n}(\FF)\simeq Gr_{k,n}(\co)$. The stabilizer subgroup of the subspace
spanned by the first $k$ vectors of the standard basis of $\FF^n$ is equal to
$$P_k(\co):=\left\{\left[\begin{array}{c|c}
                    A&B\\
                    \hline
                    0&C
                    \end{array}\right]\Big| A\in GL_k(\co),C\in GL_{n-k}(\co),B\in Mat_{k\times(n-k)}(\co) \right\}.$$
Hence $$Gr_{k,n}(\co)= GL_n(\co)/P_k(\co).$$
However, $GL_n(\co)$ is a Lie group over $\FF$ and $P_k(\co)$ is its Lie subgroup (see \cite{alesker-non-arch}, Section 6, and references therein).
By \cite{serre}, part II, Ch. IV, \S 5, the quotient space $ GL_n(\co)/P_k(\co)$ has a unique structure  of an analytic manifold (in particular, a Hausdorff locally compact space)
such that the natural quotient map $ GL_n(\co)\to GL_n(\co)/P_k(\co)$ is analytic and its differential is onto everywhere. Since $GL_n(\co)$ is compact and metrizable,
the quotient space is also compact and metrizable.
Thus, $Gr_{k,n}(\co)$ is equipped with a metrizable topology and is a compact space.

\begin{lemma}\label{L:Grass-to-finite-Grass}
For $l\geq 1$ consider the map (\ref{E:grass-map})
\begin{eqnarray}\label{E:map-theta}
Gr_{k,n}(\co)\overset{\theta_l}{\to} Gr_{k,n}(\co/\frak{m}^l)
\end{eqnarray}
which is  surjective by Lemma \ref{E:maps-grassm-local}. The map $\theta_l$ is continuous when the target space (which is a finite set) is equipped with the discrete topology.
\end{lemma}

{\bf Proof.} We need to show that the preimage under $\theta_l$ of every point in $Gr_{k,n}(\co/\frak{m}^l)$ is open in $Gr_{k,n}(\co)$.
Equivalently, we need to show that the preimage of any point in  $Gr_{k,n}(\co/\frak{m}^l)$ is open under the composition of the maps
$$GL_n(\co)\to Gr_{k,n}(\co)\overset{\theta_l}{\to}Gr_{k,n}(\co/\frak{m}^l),$$
where the first map is the quotient map. This follows from Lemma \ref{L:stabilizer} since the stabilizer $P_k(\frak{m}^l)$ is an open subgroup of $GL_n(\co)$. \qed

\item For a ring of integers $\co$ of a local field $\FF$, by Lemma \ref{E:maps-grassm-local} we have an infinite sequence of surjective maps
\begin{eqnarray}\label{E:grassmannians-seq}
Gr_{k,n}(\co/\frak{m})\overset{\iota_1}{\leftarrow} Gr_{k,n}(\co/\frak{m}^2)\overset{\iota_2}{\leftarrow} Gr_{k,n}(\co/\frak{m}^3)\overset{\iota_3}{\leftarrow} \dots.
\end{eqnarray}


These maps satisfy
$$\iota_{l-1}\circ \theta_l=\theta_{l-1}, \, l\geq 2.$$
They induce a unique map to the inverse limit
\begin{eqnarray}\label{E:map-theta0.5}
\theta\colon Gr_{k,n}(\co)\to \lim_{\underset{l}{\leftarrow}}Gr_{k,n}(\co/\frak{m}^l),
\end{eqnarray}
where the target space is equipped with the inverse limit topology.

\begin{lemma}
The map $\theta$ is a homeomorphism.
\end{lemma}

{\bf Proof.} $\theta$ is continuous since all $\theta_l$ are continuous by Lemma \ref{L:Grass-to-finite-Grass}.

 To prove injectivity,
assume that
\begin{eqnarray}\label{E:map-theta1}
\theta(\cp_1)=\theta(\cp_2).
\end{eqnarray}
This means that for any $l\geq 1$
$$\cp_1\otimes_\co \co/\frak{m}^l=\cp_2\otimes_\co \co/\frak{m}^l$$
as submodules of $(\co/\frak{m}^l)^n$. There exists $\cq_1$ such that
$$\co^n=\cp_1\oplus \cq_1.$$
By assumption, $\cp_1$ and $\cq_1$ are free of ranks $k$ and $n-k$ respectively. Hence we may assume after a change of basis of $\co^n$ that
$$\cp_1=\{(x_1,\dots,x_k,\underset{n-k\mbox{ times}}{\underbrace{0,\dots,0}})\}.$$
If there exists $X\in\cp_2\backslash \cp_1$ then there is $k<i\leq n$ such that
$$X_i\ne 0.$$
Then there exists $l>0$ such that
$$X_i \not\equiv 0 \mbox{ mod } \frak{m}^l.$$
This implies that $$X \mbox{ mod }\frak{m}^l\not\in \cp_1\otimes_\co \co/\frak{m}^l$$
in contradiction to (\ref{E:map-theta1}).
Hence $\cp_2\subset\cp_1$. By symmetry, the opposite inclusion holds too.

Let us prove that $\theta$ is onto. Let $$\{P_l\}_{l=1}^\infty \in \lim_{\underset{l}{\leftarrow}}Gr_{k,n}(\co/\frak{m}^l),$$ namely
$$\iota_{l-1}(P_l)=P_{l-1},\, l\geq 2.$$
Since $\theta_l$ is onto by Lemma \ref{E:maps-grassm-local}, there exist $\tilde P_l\in Gr_{k,n}(\co)$ such that
$$\theta_l(\tilde P_l)=P_l.$$
Since $Gr_{k,n}(\co)$ is compact and metrizable, the sequence $\{\tilde P_s\}_{s=1}^\infty$ has a subsequence converging to $\tilde P\in Gr_{k,n}(\co)$.
It is easy to see that for any $l\geq 1$
$$\theta_l(\tilde P)=P_l.$$
A continuous bijection from a compact space to a Hausdorff space is a
homeomorphism.
\qed

\end{paragraphlist}

\section{Radon transform on Grassmannians.}
\begin{paragraphlist}
\item Let $X$ be an analytic manifold over a non-Archimedean local field $\FF$, e.g., a Grassmannian $Gr_p^V$ of $p$-dimensional subspaces in an $n$-dimensional vector space $V$ over $\FF$.
\begin{definition}\label{D:smooth-functions}
A function $f\colon X\to \CC$ is called smooth if it is locally constant.
\end{definition}
The space of smooth functions is denoted by $C^\infty(X)$; it is an algebra over $\CC$.

\item Let us recall the definition of the Radon transform on Grassmannians over a non-Archimedean local field $\FF$.
Let us fix a lattice $\Lam \subset V$. Let $0\leq p\leq q\leq n$. Define the Radon transform
$$R_{pq}\colon C^\infty(Gr_q^V)\to C^\infty(Gr_p^V)$$
as follows
$$(R_{pq}f)(E)=\int_{\substack{ F\supset E\\ \dim F=q}}f(F)d\nu_E(F),$$
where the integration is over the manifold (isomorphic to $Gr_{q-p}^{V/E}$) of $q$-subspaces $F$ containing $E$ with respect to the only probability Haar
measure $\nu_E$ which is invariant under all $GL(\Lam)$-transformations preserving $E$. Note that $R_{pp}=Id$.
Clearly the operator $R_{pq}$ is $GL(\Lam)$-equivariant.

\begin{theorem}[\cite{petrov-chernov}]\label{T:radon-tr}
If $p+q=n$ then $R_{pq}\colon C^\infty(Gr_q^V)\to C^\infty(Gr_p^V)$ is an isomorphism.
\end{theorem}

\end{paragraphlist}

\section{The sine transform.}\label{S:sine}
\begin{paragraphlist}
\item Let $V$ be an  $n$-dimensional vector space over $\FF$. Let $\Lam\subset V$ be a lattice. Let $vol_\Lam$ be the Lebesgue measure on $V$ normalized so that
$$vol_\Lam(\Lam)=1.$$

Let $L, M\subset V$ be two subspaces of complementary dimensions $i$ and $n-i$ respectively. Denote
$$\Lam_L=L\cap\Lam,\, \Lam_M=M\cap\Lam.$$

In \cite{alesker-non-arch}, formula (15.2), one defined the sine of the angle between $L$ and $M$ by
\begin{eqnarray}\label{D:sine}
s(L,M):=vol_\Lam(\Lam_L+\Lam_M).
\end{eqnarray}
Clearly, $0\leq s(L,M)\leq 1$.

\item Let us give another description of $s(L,M)$. Let $$p\colon \Lam\to \Lam/\Lam_L\oplus \Lam/\Lam_M$$
be the obvious homomorphism. Let $\nu$ be any non-zero Lebesgue measure on $V/L\oplus V/M$.
\begin{proposition}\label{P:sine-another}
One has $$s(L,M)=\frac{\nu(p(\Lam))}{\nu(\Lam/\Lam_L\oplus \Lam/\Lam_M)}.$$
\end{proposition}
Clearly, $s(L,M)$ is independent of $\nu$.
This proposition follows from the following more precise result.
\begin{proposition}
There is a group isomorphism
\begin{eqnarray}\label{E:v0}
\frac{\Lam}{\Lam_L+ \Lam_M}\simeq \frac{\Lam/\Lam_L\oplus \Lam/\Lam_M}{p(\Lam)}.
\end{eqnarray}
\end{proposition}
{\bf Proof.} We will prove that both groups are isomorphic to
\begin{eqnarray}\label{E:v1}
\frac{\Lam/\Lam_L}{Im[\Lam_M\to \Lam/\Lam_L]}.
\end{eqnarray}
For the left hand side of (\ref{E:v0}) this is obvious. For the right hand side consider the homomorphism
$$F\colon \Lam/\Lam_L\to \frac{\Lam/\Lam_L\oplus \Lam/\Lam_M}{p(\Lam)}$$
given by $F(x)\equiv(x,0) \mbox{ mod }p(\Lam)$.
It is easy to see that $F$ is onto. Let us show that
$$Ker(F)=Im[\Lam_M\to \Lam/\Lam_L].$$
Let $x\in Ker (F)$. That means that $(x,0)\in p(\Lam)$. In other words there exists $\tilde y\in \Lam$ such that $x\equiv\tilde y\mbox{ mod } \Lam_L$ and $0\equiv\tilde y\mbox{ mod } \Lam_M$.
Hence $\tilde y\in \Lam_M$ and hence $x\in Im[\Lam_M\to \Lam/\Lam_L]$.
\qed

\item Let us choose $V=\FF^n,\, \Lam=\co^n$.
The sine transform\footnote{Sometimes it is also called the cosine transform.} is a linear transformation
$$\cd_{n-i}\colon C^\infty(Gr_{i,n}(\FF))\to C^\infty(Gr_{n-i,n}(\FF))$$
defined by
$$(\cd_{n-i}f)(M)=\int_{L\in Gr_{i,n}(\FF)}s(L,M)f(L) dL,$$
where $dL$ is the $GL_n(\co)$-invariant Haar probability measure on $Gr_{i,n}$. The sine transform $\cd_{n-i}$ is $GL_n(\co)$-equivariant because obviously
\begin{eqnarray}\label{E:sine-invar}
s(gL,gM)=s(L,M) \mbox{ for any } g\in GL_n(\co).
\end{eqnarray}

\begin{proposition}\label{P:6.3}[\cite{alesker-non-arch}, Proposition 8.3(2)]
The sine transform $\cd_i$ for any $1\leq i\leq n-1$ is an isomorphism of $GL_n(\co)$-representations.
\end{proposition}

\item Let us compute more explicitly $s(L,M)$ when $\dim L=1,\, \dim M=n-1$. Assume that $V=\FF^n$, $\Lam=\co^n$.
For $\bbx,\bby\in \FF^n$ denote
$$\langle \bbx,\bby\rangle_0=x_1y_1+\dots+x_ny_n.$$

A one-dimensional subspace $L$ is spanned by a vector $$\bbx=(x_1,\dots,x_n)$$
such that $|\bbx|:=\max_{i}|x_i|=1$. The vector $\bbx$ is defined uniquely up to multiplication by an element from $\co^*$.

A hyperplane $M$ can be described by a vector $\bby=(y_1,\dots,y_n)$ as
$$M=\{\bbz\in \FF^n|\, \langle \bbz,\bby\rangle_0=0\}=:\bby^{\perp_0}.$$

We can and will choose $\bby$ so that $|\bby|=1$. Under this assumption, $\bby$ is unique up to a multiplication by an element from $\co^*$.

\begin{lemma}
Let $L=span\{\bbx\}$, $M=\bby^{\perp_0}$, and $|\bbx|=|\bby|=1$. Then
$$s(L,M)=|\langle \bbx,\bby\rangle_0|.$$
\end{lemma}
{\bf Proof.} For any invertible matrix $g\in GL_n(\FF)$ one has
$$\langle g\bbx,(g^t)^{-1}\bby\rangle_0=\langle \bbx,\bby\rangle_0,$$
where $g^t$ is the transposed matrix. This and (\ref{E:sine-invar}) imply that
it suffices to prove the lemma for $\bby=(1,0,\dots,0)$. In this case,
$$\Lam_M=\{0\}\times \co^{n-1},\, \Lam_L=\co\cdot\bbx.$$
Then we have
$$\Lam_L+\Lam_M=\co\cdot\bbx+\{0\}\times \co^{n-1}=x_1\cdot\co\times \co^{n-1}.$$
Then we obtain
$$vol_\Lam(x_1\cdot\co\times \co^{n-1})=|x_1|=|\langle \bbx,\bby\rangle_0|.$$
\qed

\end{paragraphlist}

\section{Background on valuations.}\label{S:valuations}
\begin{paragraphlist}
\item Consider a complex line bundle $\cl_i$ over the Grassmannian $Gr_{i,n}(\FF)$ whose fiber over a subspace $E\in Gr_{i,n}(\FF)$ is equal to the space $D(E)$ of $\CC$-valued Lebesgue measures.
The bundle is naturally equivariant under the group $GL_n(\FF)$ of all invertible linear transformations. Let us denote by $C^\infty(Gr_{i,n}(\FF),\cl_i)$ the space of so-called smooth sections of $\cl_i$, i.e.,
sections fixed by an open subgroup of $GL_n(\FF)$. By \cite{alesker-bernstein}, Theorem 2.1,
(see also \cite{gourevitch}),  $C^\infty(Gr_{i,n}(\FF),\cl_i)$ is a $GL_n(\FF)$-irreducible representation.
We define the space of $i$-homogeneous valuations by
$$Val_i^\infty(\FF^n)=C^\infty(Gr_{i,n}(\FF),\cl_i).$$

  Define the space of valuations
$$Val^\infty(\FF^n)=\oplus_{i=0}^n Val_i^\infty(\FF^n).$$
Note that $$Val^\infty_0(\FF^n)=\CC,\, Val^\infty_n(\FF^n)=D(\FF^n).$$


\item\label{i:def-valuations} The subgroup $GL_n(\co)\subset GL_n(\FF)$ coincides with the subgroup of transformations $g$ such that $g(\co^n)=\co^n$. The line bundle $\cl_i$ is $GL_n(\co)$-equivariantly isomorphic to the trivial bundle.
Indeed, since $GL_n(\co)$ acts transitively on $Gr_{i,n}(\FF)$, it suffices to construct a non-vanishing $GL_n(\co)$-invariant section of $\cl_i$.
Such a section assigns to each subspace $E\in Gr_{i,n}(\FF)$ the Lebesgue measure $\mu_E\in D(E)$ on $E$, normalized by
$$\mu_E(\co^n \cap E)=1.$$

This induces a $GL_n(\co)$-equivariant isomorphism:
\begin{eqnarray}\label{E:imb-val-funct}
Val_i^\infty(\FF^n)\tilde\to C^\infty(Gr_{i,n}(\FF)).
\end{eqnarray}

Note that a function on $Gr_{i,n}(\FF)$ is smooth if it is invariant under an open subgroup of $GL_n(\co)$. Below we will simply identify using map (\ref{E:imb-val-funct}): 
$$Val^\infty_i(\FF^n)=C^\infty(Gr_{i,n}(\FF)).$$

\item Similarly, in \cite{alesker-non-arch} one defined the space of continuous valuations as the space of continuous sections of $\cl_i$:
\begin{eqnarray*}
Val_i(\FF^n)\tilde\to C(Gr_{i,n}(\FF)),\\
Val(\FF)=\oplus_{i=0}^n Val_i(\FF).
\end{eqnarray*}
$Val_i(V)$ and $Val(V)$ are Banach spaces, and $Val_i^\infty(V)$ and $Val^\infty(V)$ are dense in them respectively. Note that
$$Val_0(\FF)=Val_0^\infty(\FF)=\CC,\, Val_n(\FF)=Val_n^\infty(\FF)=D(\FF).$$

\item\label{i:pull-back} For a linear map $F\colon X\to Y$ between two finite dimensional $\FF$-vector spaces, the author defined in \cite{alesker-non-arch}, \S 9.7, a continuous linear map
$$F^*\colon Val_i(Y)\to Val_i(X)\mbox{ for any } i$$
called the pull-back. The construction is as follows. Let $\phi\in Val_i(Y)$. Fix an arbitrary $E\in Gr_i^X$. 
If \begin{eqnarray}\label{E:F-restriction}
F|_E\colon E\to F(E)
\end{eqnarray}
is not an isomorphism, then define 
$$(F^*\phi)(E)=0.$$
Otherwise, i.e., if \ref{E:F-restriction} is an isomorphism then define
$$(F^*\phi)(E)=(F|_E)^{-1}_*(\phi(F(E))),$$
where the lower star denotes the usual push-forward of measures.

\item $\cd_i$ is a $GL_n(\co)$-equivariant transformation. Hence its image is a $GL_n(\co)$-invariant subspace of $C^\infty(Gr_{i,n}(\FF))$.

The natural representation of $GL_n(\co)$ in the space of functions $C^\infty(Gr_{i,n}(\FF))$ is multiplicity free; for $char(\FF)=0$ it was proven first in
\cite{hill-94}, Corollary 3.2, and in general in \cite{bader-onn}.

It follows that the image of $\cd_i$ in $C^\infty(Gr_{i,n}(\FF))$, and hence $Val_i^\infty(\FF^n)$, is a (typically infinite) direct sum of certain irreducible representations of $GL_n(\co)$.

\item In \cite{alesker-non-arch}, the author defined a multiplicative structure on the space $Val^\infty(\FF^n)$, making it into a commutative associative algebra.
This algebra has a unit $1\in Val_0^\infty(\FF^n)=\CC$ and is graded:
$$Val^\infty_i(\FF^n)\cdot Val^\infty_j(\FF^n)\subset Val^\infty_{i+j}(\FF^n).$$
The group $GL_n(\FF)$ acts naturally on $Val^\infty(\FF^n)$ by algebra automorphisms.

We will not discuss the general construction of the product here. However, we will need several specific results from \cite{alesker-non-arch} in Section \ref{S:form-val}, where they will be stated explicitly.

\item Let $V$ be an $n$-dimensional vector space over $\FF$. Let $\Lam\subset V$ be a lattice. Let us denote by $V_{\Lam,i}\in Val_i^\infty(V)$ the only element such that
\begin{eqnarray}\label{E:V-Lam-i}
(V_{\Lam,i}(E))(E\cap \Lam)=1 \mbox{ for all } E\in Gr_i(V).
\end{eqnarray}
Clearly, $V_{\Lam,i}$ is $GL(\Lam)$-invariant.

\end{paragraphlist}

\section{Explicit description of projective space over a local ring.}\label{S:expl-local-ring}
\begin{paragraphlist}
\item The material presented in this section is not new and is part of the folklore.

Let $(A,\frak{p})$ be a local ring. In this section, we present a more explicit description of the Grassmannian of lines
$Gr_{1,n}(A)$ which one denotes by $\PP^{n-1}(A)$ and calls a projective space.

Consider the subset of $A^n$ which we call a sphere
$$\SS^{n-1}(A)=\{(x_1,\dots, x_n)|\, \exists i=1,\dots,n\mbox{ s.t. } x_i\in A^*\}.$$
\begin{lemma}\label{L:submodule1}
If $\bbv\in \SS^{n-1}(A)$ then $A\cdot \bbv\in \PP^{n-1}(A)$.
\end{lemma}

{\bf Proof.} The $A$-submodule $A\cdot\bbv\subset A^n$ is free of rank 1. Indeed, the map $A\to A\cdot\bbv$ given by
$$x\mapsto x\cdot \bbv$$
is an isomorphism.

Let us assume for simplicity of the notation that $v_1\in A^*$. Set
$$\cq=\{(0,x_2,\dots,x_n)| x_i\in A \mbox{ for } i=2,\dots, n\}.$$
Clearly, $\cq$ is isomorphic to $A^{n-1}$ and
$$A^n=A\cdot\bbv\oplus \cq.$$
Hence $A\cdot\bbv\in \PP^{n-1}(A)$. \qed

\begin{lemma}\label{L:submodule2}
\begin{enumerate}
\item\label{i:pr-sp1} The map $\SS^{n-1}(A)\to \PP^{n-1}(A)$ given by
\begin{eqnarray}\label{E:map-sphere-proj}
\bbv\mapsto A\cdot \bbv
\end{eqnarray}
is onto.

\item\label{i:pr-sp2} Two vectors $\bbv,\bbw\in \SS^{n-1}(A)$ have the same image under the map (\ref{E:map-sphere-proj}) if and only if there is  $\alp\in A^*$ such that
$$\bbw=\alp\cdot\bbv.$$
Such $\alp$ is necessarily unique.
\end{enumerate}
\end{lemma}
{\bf Proof.} Part \ref{i:pr-sp1}. Assume that $$A^n=\cp\oplus \cq,$$
where $\cp\simeq A,\, \cq\simeq A^{n-1}$. Let $\bbv\in \cp$ be a basis of $\cp$.  Let us prove that $\bbv\in \SS^{n-1}(A)$.
Choose a basis $\bm{\xi}_1,\dots,\bm{\xi}_{n-1}$ of $\cq$. Then
$$\bbv, \bm{\xi}_1,\dots,\bm{\xi}_{n-1}$$
is a basis of $A^{n}$. Let us denote their images in $A^n\otimes_A A/\frak{p}=\kappa^n$ by
$$\tilde\bbv, \tilde{\bm{\xi}}_1,\dots,\tilde{\bm{\xi}}_{n-1}.$$
The latter elements form a basis of $\kappa^n$. In particular, $\tilde\bbv\ne 0$ in $\kappa^n$. That means that
$$\bbv\in \SS^{n-1}(A).$$

Part \ref{i:pr-sp2}. Assume that $A\cdot\bbv=A\cdot\bbw$ where $\bbv,\bbw\in \SS^{n-1}(A)$. We may assume without loss of generality that
$v_1\in A^*$. Then necessarily $w_1\in A^*$. Multiplying $\bbv$ and $\bbw$ by elements from $A^*$ we may assume that
$$v_1=w_1=1.$$
Let us prove that $\bbv=\bbw$. We have
$$A\cdot(1,v_2,\dots,v_n)=A\cdot(1,w_2,\dots,w_n).$$
In particular $$(1,v_2,\dots,v_n)=\gamma(1,w_2,\dots,w_n),$$
where $\gamma\in A$. Hence $\gamma=1$. Hence $\bbv=\bbw$. \qed

\item Let now $A=\co$ be the ring of integers in a non-Archimedean local field $\FF$. By Lemma \ref{L:two-Grassmannians}, the natural map
$$\PP^{n-1}(\co)\to \PP^{n-1}(\FF)$$
given by $\cp\mapsto \cp\otimes_\co\FF$ is a $GL_n(\co)$-equivariant bijection. Below we will identify these two spaces via this map.
This space will be denoted by $\PP^{n-1}$.

The sphere $\SS^{n-1}(\co)$ will be denoted by $\SS^{n-1}$. Obviously, it can be described as follows:
\begin{eqnarray}\label{E:sphere-non-a}
\SS^{n-1}=\{\bbx\in \FF^n|\, |\bbx|=1\}.
\end{eqnarray}
The natural action of the group $GL_n(\co)$ in $\SS^{n-1}$ is transitive.

\item By Lemma \ref{L:submodule2}, we have the natural $GL_n(\co)$-equivariant map onto
$$\SS^{n-1}\to \PP^{n-1}.$$
Clearly the push-forward of the $GL_n(\co)$-invariant probability Haar measure on $\SS^{n-1}$ under this map is equal to the $GL_n(\co)$-invariant probability Haar measure on $\PP^{n-1}$.
Hence the integral of any function on $\PP^{n-1}$ with respect to this Haar measure is equal to the integral of its pull-back to the sphere $\SS^{n-1}$ with respect to the corresponding Haar measure.
We will use this property in computations.

Hence, let us describe explicitly the Haar measure on $\SS^{n-1}$.  By (\ref{E:sphere-non-a}), $\SS^{n-1}$ is an open compact subset of $\FF^n$. Let $d^nx$ denote the Lebesgue measure on $\FF^n$.
Recall that $\int_{\co^n} d^nx=1$.
It is preserved by the group $GL_n(\co)$. Hence its restriction to $\SS^{n-1}$ is non-zero and invariant. Although $d^nx$ is not a probability measure (the total measure of the sphere equals $1-q^{-n}$),
 we will use it in computations.

\end{paragraphlist}

\section{Harmonic analysis on projective space.}\label{S:harmonic-projective}
\begin{paragraphlist}
\item  In this section, we review some results due to P. Humphries \cite{humphries}. Assume that $n\geq 2$.
Let us consider the space $\FF^n$ and the lattice $\co^n$. We denote the Grassmannian $Gr_{1}(\FF^n)$, which is a projective space, by $\PP^{n-1}$.

The group $G:=GL_n(\co)$ acts on $\PP^{n-1}$. Let us denote by
$$P:=Stab_{GL_n(\co)}(span(\bbe_n))=\left\{\left[\begin{array}{c|c}
                                      A&0\\
                                      \hline
                                      c&d
                                   \end{array}\right]\in GL_n(\co)\right\}$$
the stabilizer of the line $span(\bbe_n)$. Here $A\in GL_{n-1}(\co),\, d\in \co^*$.
Then $$\PP^{n-1}=G/P.$$

\item The action of $G$ on $\PP^{n-1}$ commutes with the maps (\ref{E:grassmannians-seq}), and induces a sequence of $G$-equivariant imbeddings
$$\CC\inj C(\PP^{n-1}(\co/\frak{m}))\inj C(\PP^{n-1}(\co/\frak{m}^2))\inj C(\PP^{n-1}(\co/\frak{m}^3))\inj \dots,$$
where all spaces are finite dimensional.

Moreover, there is a unique $G$-invariant probability measure on $\PP^{n-1}$. It defines a $G$-invariant scalar product on $C^\infty(\PP^{n-1})$.

 Let us define $\rho_0\simeq \CC$ to be the subspace of constant functions on $\PP^{n-1}$.

 Define
 $$\rho_1=C(\PP^{n-1}(\co/\frak{m}^1))\ominus\rho_0,$$
 and for $m\geq 2$, define
$$\rho_m:=C(\PP^{n-1}(\co/\frak{m}^m))\ominus C(\PP^{n-1}(\co/\frak{m}^{m-1})).$$
Since every smooth function on $\PP^{n-1}(\FF)=\PP^{n-1}(\co)$ factors through $\PP^{n-1}(\co/\frak{m}^l)$ for some $l$, we get a decomposition into an infinite orthogonal direct sum of $G$-invariant subspaces
$$C^\infty(\PP^{n-1})=\oplus_{m=0}^\infty \rho_m.$$
By \cite{humphries}, Theorem 2.16, the $G$-modules $\rho_m$ are irreducible and pairwise non-isomorphic. Moreover, all $\rho_m$ have pairwise different dimensions (\cite{humphries}, Lemma 2.15).

\item Since $C^\infty(\PP^{n-1})=C^\infty(G/P)$ is a multiplicity free $G$-module, every $\rho_m,\, m\geq 0,$ has a unique, up to a proportionality, $P$-invariant vector $P_m$.
We will need its explicit construction due to Humphries \cite{humphries}, Proposition 3.4. It says that
\begin{eqnarray*}
P_0=1,\\
P_1(x_1,\dots,x_n)=\left\{\begin{array}{ccc}
                          1&\mbox{ if }&\max\{|x_1|,\dots,|x_{n-1}|\}\leq q^{-1},\\
                          -\frac{q-1}{q(q^{n-1}-1)}&\mbox{ if }&\max\{|x_1|,\dots,|x_{n-1}|\}=1,
                          \end{array}\right.
\end{eqnarray*}
and for $m\geq 2$
\begin{eqnarray*}
P_m(x_1,\dots,x_n)=\left\{\begin{array}{ccc}
                          1&\mbox{ if }&\max\{|x_1|,\dots,|x_{n-1}|\}\leq q^{-m},\\
                          -\frac{1}{q^{n-1}-1}&\mbox{ if }&\max\{|x_1|,\dots,|x_{n-1}|\}=q^{-m+1},\\
                          0& &\mbox{ otherwise},
                          \end{array}\right.
\end{eqnarray*}
where we assume that $\max_{1\leq i\leq n}|x_i|=1$.

\end{paragraphlist}

\section{Harmonic analysis on dual projective space.}\label{S:dual-proj-space}
\begin{paragraphlist}
\item
As in Section \ref{S:harmonic-projective}, $\PP^{n-1}=G/P$ denotes the projective space of lines in $\FF^n$. Let $\PP^\vee$ be the dual projective space of hyperplanes in $\FF^n$.
Then $$\PP^{\vee} =G/\bar P,$$
where $\bar P$ is the transpose of $P$. We will identify
$$\PP^\vee \tilde\to \PP^{n-1}$$
using the annihilator with respect to the bilinear form $\langle\cdot,\cdot\rangle_0$, namely a hyperplane $H$ is identified with the line in $\FF^n$
$$H^{\perp_0}:=\{\bbx\in\FF^n|\, \langle\bbx,\bby\rangle_0=0 \mbox{ for all } \bby\in H\}.$$

In order for this identification to be $G$-equivariant, one has to change the action of $G$ on $\PP^{n-1}$ to the following one:
\begin{eqnarray}\label{E:G-action}
(g,l)\mapsto (g^t)^{-1}l.
\end{eqnarray}
We will denote by $\bar\PP^{n-1}$ the space $\PP^{n-1}$  with this modified
$G$-action.

\item The Radon transform $$R_{1,n-1}\colon C^\infty(\bar\PP^{n-1})\to C^\infty(\PP^{n-1})$$
is a $G$-equivariant isomorphism by \cite{petrov-chernov}. Below we will abbreviate
$$R:=R_{1,n-1}.$$
 The inverse image
$$\bar\rho_m:=R^{-1}(\rho_m)$$
 is isomorphic to $\rho_m$ via $R$. Hence
$$C^\infty(\bar \PP^{n-1})=\oplus_{m=0}^\infty \bar\rho_m,\, \bar\rho_m\simeq \rho_m.$$
By \cite{humphries}, Lemma 2.15, all $\rho_m$ have different dimensions for different $m\geq 0$. Hence, $\bar\rho_m$ has to coincide,
as a subspace of functions on $\PP^{n-1}$, with $\rho_m$.  Let us summarize this as follows.
\begin{lemma}
The $G$-representations $\rho_m$ and $\bar\rho_m$ are isomorphic for any $m\geq 0$. An isomorphism is given by the Radon transform $R$.
\end{lemma}

\item Let us denote for brevity $\PP=\PP^{n-1}$ and $\bar\PP=\bar\PP^{n-1}$.
  Let us construct $P$-invariant functions $\bar P_m\in \bar\rho_m\subset C^\infty(\bar\PP)$ similar to Humphries' functions $P_m\in C^\infty(\PP)$.

Note that $P_m\in C(P\backslash G/P)$ and $\bar P_m\in C(\bar P\backslash G/P)$ when the actions of $G$ on $P$ and $\bar P$ are the standard ones.

\begin{proposition}\label{bar-P-m}
Let $\bbx=(x_1,\dots,x_n)\in\SS^{n-1}$. Then $\bar P_0=1$, and for $m\geq 1$ one has
$$\bar P_m(\bbx)=\left\{\begin{array}{ccc}
                      1&\mbox{ if }& |x_n|\leq q^{-m},\\
                      -\alp_{n,m}&\mbox{ if }& |x_n|=q^{-(m-1)}\\
                      0&\mbox{ otherwise }&
                        \end{array}\right. ,$$
                        where $$\alp_{n,m}=\left\{\begin{array}{ccc}
                               \frac{1-q^{-(n-1)}}{q-1}& \mbox{ if }& m=1,\\
                               \frac{1}{q-1}&\mbox{ if }& m\geq 2.
                              \end{array}\right. $$
\end{proposition}
{\bf Proof.}  \underline{Step 1.} It is clear that $\bar P_m$ are $P$-invariant (for the action \ref{E:G-action}).

\underline{Step 2.} Let us prove that $\bar P_m$ with $m\geq 1$ are orthogonal to $\bar P_0=1$, i.e.,
\begin{eqnarray}\label{E:integ-vanish}
\int_{\SS^{n-1}}\bar P_m(\bbx) d^nx=0.
\end{eqnarray}
We denote by $vol_r$ the Lebesgue measure on $\FF^r$ normalized so that $vol_r(\co^r)=1$.

For $m\geq 2$ one has
\begin{eqnarray*}
\int_{\SS^{n-1}}\bar P_m(\bbx) d^nx=\int_{|x_n|\leq q^{-(m-1)}}\bar P_m(\bbx)d^nx=\\
vol_{n-1}(\SS^{n-2})vol_1(\pi^m\co)-\frac{1}{q-1} vol_{n-1}(\SS^{n-2})vol_1(\pi^{m-1}\co^*)=\\
vol_{n-1}(\SS^{n-2})\left(q^{-m}-q^{-(m-1)}\frac{1-q^{-1}}{q-1}\right)=0.
\end{eqnarray*}

For $m=1$ one has
\begin{eqnarray*}
\int_{\SS^{n-1}}\bar P_1(\bbx) d^nx=\\
vol_{n-1}(\SS^{n-2})\cdot vol_1(\pi \co)-\frac{1-q^{-(n-1)}}{q-1} vol_{n-1}(\co^{n-1})\cdot vol_1(\co^*)=\\
(1-q^{-(n-1)})q^{-1}-\frac{1-q^{-(n-1)}}{q-1}  (1-q^{-1})=0.
\end{eqnarray*}

\underline{Step 3.} Let us show that for any $m\geq 1$
$$\bar P_m\in \bar\rho_m\subset C^\infty(\bar\PP^{n-1}).$$
For $m=1$, the definition of $\bar P_1$ shows that it factors through reduction
modulo $\frak{m}$; (\ref{E:integ-vanish}) shows that it is orthogonal to the
constants.

 For $m\geq 2$ let $$\theta_{m-1}\colon\bar \PP^{n-1}\to \bar\PP^{n-1}(\co/\frak{m}^{m-1})$$
be the natural map (\ref{E:map-theta}). It is $G$-equivariant, and its target is a finite set. The pull-back induced by this map is  a $G$-equivariant map on function spaces
\begin{eqnarray}\label{E:pull-back-finite}
C(\bar\PP^{n-1}(\co/\frak{m}^{m-1}))\to C^\infty(\bar \PP^{n-1}).
\end{eqnarray}
In order to show that $\bar P_m\in \bar\rho_m\subset C^\infty(\bar\PP^{n-1})$
it suffices to show that $\bar P_m$ is orthogonal to the image of the map (\ref{E:pull-back-finite}) and factors through reduction modulo $\frak{m}^m$. The latter condition follows directly from
the definition of $\bar P_m$. Let us check the former condition.

Let $m\geq 2$. Since $(x_1,\dots,x_n)\in \SS^{n-1}$ by assumption,
the pull-back of any function $f\in C(\bar\PP^{n-1}(\co/\frak{m}^{m-1}))$ depends only on the class of $(x_1:\dots: x_n)$ modulo $\frak{m}^{m-1}$. In particular, it is constant on any subset of the form
$$\{(a_1:\dots :a_{n-1}:x_n)|\, x_n\in \pi^{m-1}\co\},$$
where $a_1,\dots,a_{n-1}$ are fixed.

But $\bar P_m(\bbx)=\bar P_m(|x_n|)$ is supported on the set
$$ \{(x_1:\dots :x_{n-1}:x_n)|\, x_n\in \pi^{m-1}\co\}=(\SS^{n-2}\times \pi^{m-1}\co)/\co^*.$$
Hence
$$\int_{\SS^{n-1}}f\cdot \bar P_m d^nx=\int_{\SS^{n-2}}dx_1\dots dx_{n-1}\int_{\pi^{m-1}\co}dx_n f(\bbx)\cdot \bar P_m(|x_n|) \overset{(\ref{E:integ-vanish})}{=}0.$$
\qed

\end{paragraphlist}

\section{Hodge-Riemann relations and harmonic analysis on Grassmannians.}\label{S:HR-grassmannians}\label{S:form-val}
\begin{paragraphlist}
\item In this section, we reformulate the Hodge-Riemann relations in terms of analysis on Grassmannians in a way similar to Kotrbat\'y's approach \cite{kotrbaty-conj} for even valuations on convex sets.

\item Let us fix a lattice $\Lam\subset V$. For any subspace $M\subset V$ denote by $vol_M$ the Lebesgue measure on $M$ such that $vol_M(M\cap \Lam)=1$.
 This induces a trivialization of the vector bundle $\cl_i$ over $Gr_i^V$: indeed the fiber $\cl_i|_M=D(M)=\CC\cdot vol_M$. Thus we will identify the space $Val_i^\infty(V)$
 with a subspace of locally constant functions on $Gr_i^V$.

  Consider valuations $\phi\in Val_i^\infty(V),\psi\in Val_{n-i}^\infty(V)$ given by
\begin{eqnarray}\label{E:phi-val}
\phi=\int_{Gr_{n-i}^{V}}\hat f(L)p^*_L(\mu_{L}) dL,\\\label{E:phi-val2}
\psi=\int_{Gr_{i}^{V}}\hat g(M)p^*_M(\mu_{M}) dM,
\end{eqnarray}
where $$\hat f\in C^\infty(Gr_{n-i}^V),\, \hat g\in C^\infty(Gr_i^V),$$
the integrations are with respect to the $GL(\Lam)$-invariant Haar probability measures on the Grassmannians, $p_L\colon V\to V/L$ is the canonical quotient map,
and $\mu_{L}$ is the Lebesgue measure on $V/L$ normalized so that $\mu_{L}(p_L(\Lam))=1$, and similarly for $p_M, \mu_{M}$. The upper star in the above formula denotes the pull-back on valuations
defined in Section \ref{S:valuations}, paragraph \ref{i:pull-back}.
In other words
\begin{eqnarray}\label{E:val-via-sine}
\phi=\cd_{i}\hat f,\\
\psi=\cd_{n-i} \hat g.
\end{eqnarray}

Below we denote by $(u,v)$ the scalar product of two functions on a Grassmannian
$$(u,v):=\int_{Gr_p^V}\bar u(M) v(M) dM,$$
where the integration is with respect to the $GL(\Lam)$-invariant Haar probability measure.

\begin{lemma}\label{L:producvt-two-val}
$$\frac{\phi\cdot \psi}{vol_\Lam}=\int_{Gr_{n-i}^{V}}\int_{Gr_{i}^{V}}\hat  f(L)\hat g(M)s(L,M) dL\cdot dM=(\bar{\hat f}, \cd_{n-i}\hat g)=(\cd_i\bar{\hat f},\hat g),$$
where $vol_\Lam$ is the Lebesgue measure on $V$ normalized so that $vol_\Lam(\Lam)=1$, and $(\cdot,\cdot)$ denotes the Hermitian product on  functions on corresponding Grassmannians.
\end{lemma}
{\bf Proof.} Denote $p_L\times p_M\colon V\to V/L\times V/M$. Below we denote by $\mu_L\boxtimes \mu_M$ the usual product measure, and $(p_L\times p_M)^*$ the pull-back map on valuations.
 By \cite{alesker-non-arch}, formula (14.1), we have
\begin{eqnarray*}
\frac{\phi\cdot \psi}{vol_\Lam}=\int_{Gr_{n-i}^{V}}dL\int_{Gr_{i}^{V}}dM \hat f(L)\hat g(M)\frac{(p_L\times p_M)^*(\mu_L\boxtimes \mu_M)}{vol_\Lam}\overset{\mbox{Prop. } \ref{P:sine-another}}{=}\\
\int_{Gr_{n-i}^{V}}dL\int_{Gr_{i}^{V}}dM \hat f(L)\hat g(M)s(L,M).
\end{eqnarray*}
\qed

\item Assume $\phi\in Val_i^\infty(V)$ is given by (\ref{E:phi-val}).
\begin{lemma}[Lemma 16.2, \cite{alesker-non-arch}]\label{L:product-intr-vol}
Let $\phi\in Val_i^\infty(V)$ be given by (\ref{E:phi-val}).
Assume $0\leq k\leq n-i$. Then, identifying $Val_{i+k}^\infty(V)$ with the space of smooth functions on $Gr_{i+k}^V$ using lattice $\Lam$, one has
$$V_1^k\cdot \phi=c_{n,i,k}(\cd_{i+k}\circ R_{n-i-k, n-i})\hat f,$$
where $c_{n,i,k}>0$ is a constant and $V_1:=V_{\Lam,1}$.

In particular, if $k=n-2i$ one has
$$V_1^{n-2i}\cdot \phi=c_{n,i}\cd_{n-i}( R_{i, n-i}\hat f),$$
where $c_{n,i}>0$.
\end{lemma}

\item For $i\leq n/2$ define the Hermitian form on $Val_i^\infty(V)$.
\begin{eqnarray}\label{E:herm-form}
Q(\phi)=(-1)^{i} \frac{\bar\phi\cdot \phi\cdot V_1^{n-2i}}{vol_\Lam}.
\end{eqnarray}
Lemmas \ref{L:product-intr-vol} and \ref{L:producvt-two-val} immediately imply
\begin{lemma}\label{L:form-necessary}
One has
$$Q(\phi)=(-1)^{i}c_{n,i}(\hat f, \cd_{n-i}(R_{i,n-i}\hat f)).$$
\end{lemma}

\item The (non-mixed) Hodge-Riemann relations say that if $i\leq \frac{n}{2}$ and $\phi\in P_i\subset Val_i^\infty(V)$ is a non-zero valuation then
$$Q(\phi)>0,$$
where $$P_i=\{\phi\in Val_i^\infty(V)|\, \phi\cdot (V_1)^{n-2i+1}=0\}$$
is the subspace of primitive valuations.

Let us rewrite this conjecture in an analytic language. We assume that $V=\FF^n$ and $\Lam=\co^n$. Set $\cp_i=\cd_i^{-1}(P_i)\subset C^\infty(Gr_{n-i,n}(\FF))$.
It is a $GL_n(\co)$-invariant subspace of $C^\infty(Gr_{n-i,n}(\FF))$.

Then Lemma \ref{L:form-necessary} and Proposition \ref{P:6.3} imply that the Hodge-Riemann relations in degree $i$ are equivalent to
\begin{conjecture}\label{HR-conj1}
For $\hat f\in \cp_i\backslash\{0\}$
$$(-1)^i(\hat f,(\cd_{n-i}\circ R_{i,n-i})\hat f)>0.$$
\end{conjecture}
We can reformulate the last conjecture a bit further. The representation of $GL_n(\co)$ in the space $C^\infty(Gr_{n-i,n}(\FF))$ is multiplicity free, and this space is a direct sum
of irreducible finite dimensional representations. The operator $(-1)^i\cd_{n-i}\circ R_{i,n-i}$ is $GL_n(\co)$-equivariant, hence it maps each irreducible component $\rho$ to itself, and by Schur's lemma
its restriction to $\rho$ is proportional to the identity operator. Let us denote by $\lam_\rho$ this constant of proportionality. Then Conjecture \ref{HR-conj1} is equivalent to
\begin{conjecture}\label{HR-conj2}
Let $\rho\subset C^\infty(Gr_{n-i,n}(\FF))$ be an irreducible $GL_n(\co)$-representation satisfying
$$ \cd_i(\rho)\subset P_i.$$ Then $\lam_\rho>0.$
\end{conjecture}

\item Let us consider the case $i=1$ more explicitly. We have
$$P_1=\{\phi\in Val^\infty_1(\FF^n)|\, \phi\cdot (V_1)^{n-1}=0\}.$$

\begin{lemma}
Under the isomorphism (\ref{E:imb-val-funct})
$$Val_1^\infty(\FF^n)\tilde\to C^\infty(\PP)$$
the subspace $P_1$ is identified with the subspace of functions
$$\{f\in C^\infty(\PP)|\, \int_{\PP}f(L)dL=0\},$$
where $dL$ is the $GL_n(\co)$-invariant Haar probability measure on $\PP$.
\end{lemma}
{\bf Proof.}  The map $C^\infty(\PP)\simeq Val_1^\infty(\FF^n)\to \CC$ given by
$$\phi\mapsto \frac{\phi\cdot (V_1)^{n-1}}{vol},$$
where $vol$ is any non-zero Lebesgue measure, is a $GL_n(\co)$-equivariant linear functional whose kernel is equal to $P_1$. This functional does not vanish on the constant function 1 because $(V_1)^n\ne 0$ by
\cite{alesker-non-arch}, Proposition 14.3. However, by $GL_n(\co)$-equivariance, it must vanish on any other $GL_n(\co)$-irreducible component. The sum of all those components is precisely equal to the orthogonal complement of the constant function,
which is the space of functions with vanishing integral. \qed

\hfill

 Thus Conjecture \ref{HR-conj2} in degree $i=1$ is equivalent to the following statement which is the main result of this paper.
 \begin{theorem}\label{T:main-HR1}
 The restriction of the operator
 $$\cd_{n-1}\circ R_{1,n-1}\colon C^\infty(Gr_{n-1,n}(\FF))\to C^\infty(Gr_{n-1,n}(\FF))$$
 to the subspace of functions with vanishing integral is negative definite, i.e., $\lam_\rho>0$ for any irreducible $GL_n(\co)$-submodule $\rho\subset C^\infty(Gr_{n-1,n}(\FF))$ different from
 the trivial representation.
 \end{theorem}

\end{paragraphlist}

\section{The Radon transform.}
\begin{paragraphlist}
\item Consider the Radon transform
$$R\colon C^\infty(\bar\PP)\to C^\infty(\PP).$$
Since this map is $G$-equivariant and maps isomorphically $\bar\rho_m$ to $\rho_m$, there are constants $\alp_m$, $m\geq 0$ such that
$$R(\bar P_m)=\alp_m P_m.$$
\begin{proposition}\label{radon-sign}
For each $m\geq 0$ one has $$\alp_m>0.$$
\end{proposition}
{\bf Proof.} For $m=0$ the result is obvious. Hence assume that $m\geq 1$. Then
$$P_m(\bbe_n)=1.$$
Hence it suffices to show that
\begin{eqnarray}\label{E:radon-ineq}
(R(\bar P_m))(\bbe_n)>0.
\end{eqnarray}
We have
\begin{eqnarray*}
(R(\bar P_m))(\bbe_n)=\int_{\PP(\{x_n=0\})}\bar P_m((x_1:\dots:x_{n-1}:0))d\nu_{e_n}(x)>0
\end{eqnarray*}
since the expression under the integral equals 1. \qed

\end{paragraphlist}

\section{The sine-transform.}
\begin{paragraphlist}
\item Consider the sine-transform $\cd:=\cd_{n-1}$
$$\cd\colon C^\infty(\PP)\to C^\infty(\bar \PP).$$
This is a $G$-equivariant isomorphism. Hence there are constants $\beta_m\in \CC$, $m\geq 0$, such that
$$\cd(P_m)=\beta_m \bar P_m.$$

\begin{proposition}\label{cosine-sign}
For $m\geq 1$ one has $$\beta_m<0,$$
while $\beta_0>0$.
\end{proposition}
{\bf Proof.} For $m=0$ the statement is obvious. Let us assume that $m\geq 1$. One has
$$\bar P_m(\bbe_1)=1.$$
Hence it suffices to show that
\begin{eqnarray}\label{E:cosine-ineq}
(\cd P_m)(\bbe_1)<0.
\end{eqnarray}
For $\bbx=(x_1,\dots,x_{n-1},x_n)\in\FF^n$, let us denote $X:=(x_1,\dots,x_{n-1})$. Thus $\bbx=(X,x_n)$.

First let us assume that $m\geq 2$. One has
\begin{eqnarray*}
(1-q^{-n})(\cd P_m)(\bbe_1)=\int_{\SS^{n-1}}|\langle\bbe_1,\bbx\rangle_0| P_m(\bbx)d^nx=\int_{\SS^{n-1}}|x_1| P_m(\bbx)d^nx=\\
\int_{x_n\in\co^*}dx_n\left(\int_{|X|\leq q^{-m}}|x_1| d^{n-1}X -\frac{1}{q^{n-1}-1}\int_{|X|=q^{-(m-1)}}|x_1|d^{n-1}X \right)=\\
vol_1(\co^*)\left(q^{-mn}\int_{\co^{n-1}}|x_1|d^{n-1}X-\frac{q^{-(m-1)n}}{q^{n-1}-1}\int_{X\in \SS^{n-2}}|x_1|d^{n-1}X\right)=\\
vol_1(\co^*)q^{-(m-1)n}\left(q^{-n}\int_{x\in\co}|x|dx-   \frac{1}{q^{n-1}-1} \int_{X\in \SS^{n-2}}|x_1|d^{n-1}X \right)=\\
vol_1(\co^*)q^{-(m-1)n}\left(\frac{q^{-n}}{1+q^{-1}}-\frac{1}{q^{n-1}-1} \int_{X\in \SS^{n-2}}|x_1|d^{n-1}X \right).
\end{eqnarray*}

Let us compute the integral in brackets
\begin{eqnarray*}
\int_{X\in \SS^{n-2}}|x_1|d^{n-1}X.
\end{eqnarray*}

\underline{Case 1.} Assume $n=2$. Then this integral is equal to $vol_1(\co^*)=1-q^{-1}$. Hence in this case
\begin{eqnarray*}
(1-q^{-2})(\cd P_m)(\bbe_1)=vol_1(\co^*)q^{-2(m-1)}\left(\frac{q^{-2}}{1+q^{-1}}-\frac{1}{q-1}(1-q^{-1})\right)<0,
\end{eqnarray*}
as required.

\underline{Case 2.} Assume $n>2$. Then we have
\begin{eqnarray*}
\int_{X\in \SS^{n-2}}|x_1|d^{n-1}X=\int_{|x_1|=1}+\int_{|x_1|<1}=\\
\int_{\co^*\times\co^{n-2}}d^{n-1}x+\int_{x_1\in\pi\co}|x_1|dx_1\cdot vol_{n-2}(\SS^{n-3})=\\
vol_1(\co^*)+q^{-2}\left(\int_{\co}|x|dx\right)(1-q^{-(n-2)})=\\
1-q^{-1}+\frac{q^{-2}}{1+q^{-1}}(1-q^{-(n-2)}).
\end{eqnarray*}
Then we obtain
\begin{eqnarray*}
(1-q^{-n})(\cd P_m)(\bbe_1)=\\
vol_1(\co^*)q^{-(m-1)n}\left(\frac{q^{-n}}{1+q^{-1}}-\frac{1}{q^{n-1}-1}\left(1-q^{-1}+\frac{q^{-2}}{1+q^{-1}}(1-q^{-(n-2)})\right)\right)=\\
-vol_1(\co^*)q^{-(m-1)n}\left(\frac{q-1}{(q+1)(q^{n-1}-1)}\right)<0.
\end{eqnarray*}

Thus, we have shown so far that for any $n\geq 2$ and $m>1$ one has
$$\beta_m<0.$$
It remains to show that for $n\geq 2$
$$\beta_1<0.$$
Since
$$\bar P_1(\bbe_n)<0,$$
it suffices to show that
$$(\cd P_1)(\bbe_n)>0.$$
 One has
\begin{eqnarray*}
(1-q^{-n})(\cd P_1)(\bbe_n)=\int_{\SS^{n-1}}|x_n| P_1(\bbx)d^nx=\int_{|X|<1}+\int_{|X|=1}=\\
\int_{\pi \co^{n-1}\times\co^*}d^nx -\frac{q-1}{q(q^{n-1}-1)}\int_{\SS^{n-2}\times \co}|x_n|d^nx=\\
vol_{n-1}(\pi\co^{n-1})\cdot vol_1(\co^*) -\frac{q-1}{q(q^{n-1}-1)}vol_{n-1}(\SS^{n-2})\int_{\co}|x|dx=\\
q^{-(n-1)}(1-q^{-1})-\frac{q-1}{q(q^{n-1}-1)}(1-q^{-(n-1)})\frac{1}{1+q^{-1}}=\\
\frac{q^{-n}(q-1)}{q+1}>0,
\end{eqnarray*}
as required. \qed

\end{paragraphlist}

\section{Proof of Hodge-Riemann relations in degree 1.}
We have seen in Section \ref{S:HR-grassmannians} that the (non-mixed) Hodge-Riemann relations in degree 1 are equivalent to Theorem \ref{T:main-HR1}.  But this immediately follows from Propositions \ref{radon-sign} and \ref{cosine-sign}.

\end{document}